\documentclass{amsart}

\usepackage{amsmath}
\usepackage{amsthm}
\usepackage{amsfonts}

\usepackage[margin=1.55in]{geometry}
\usepackage{verbatim}
\newtheorem{theorem}{Theorem}
\newtheorem{lemma}[theorem]{Lemma}
\newtheorem{cor}[theorem]{Corollary}

\newtheorem{obs}[theorem]{Observation}
\theoremstyle{definition}

\usepackage{url}
\DeclareMathOperator\lcm{\mathrm{lcm}}
\newcommand\newk{\tilde k^*}
\newcommand\ratio{\beta}
\newcommand\numerator{\mu}
\newcommand{\floor}[1]{\left\lfloor #1 \right\rfloor}

\title[Lower Bounds  for Least Common Multiples of Arithmetic Progressions]{Asymptotic Improvements of Lower Bounds\\for the Least Common Multiples\\of Arithmetic Progressions}

\author{Daniel M. Kane}
\address{Department of Mathematics, Stanford University\newline\indent Building 380, Sloan Hall \newline\indent Stanford, CA 94305}
\email{dankane@math.stanford.edu, aladkeenin@gmail.com}

\author{Scott Duke Kominers}
\address{
Society of Fellows, Department of Economics,
Program for Evolutionary Dynamics, \newline\indent  and Center for Research on Computation and Society,  Harvard University, \newline\indent and Harvard Business School
One Brattle Square, Suite 6\newline\indent
Cambridge, MA 02138-3758}
\email{kominers@fas.harvard.edu, skominers@gmail.com}


\begin{document}
\begin{abstract}

For relatively prime positive integers $u_0$ and $r$, we consider the least common multiple $L_n:=\lcm(u_0,u_1,\ldots, u_n)$ of the finite arithmetic progression $\{u_k:=u_0+kr\}_{k=0}^n$.  We derive new lower bounds on $L_n$ which improve upon those obtained previously when either $u_0$ or $n$ is large.  When $r$ is prime, our best bound is sharp up to a factor of $n+1$ for~$u_0$ properly chosen, and is also nearly sharp as $n\to\infty$.
\end{abstract}
\maketitle

\section{Introduction}\label{intro}

The search for effective bounds on the least common multiples of arithmetic progressions began with the work of Hanson~\cite{Ha} and Nair~\cite{Na}, who respectively found upper and lower bounds for $\lcm(1,\ldots, n)$.  Decades later, Bateman, Kalb, and Stenger~\cite{BKS} and Farhi~\cite{BF1} respectively obtained asymptotics and  nontrivial lower bounds for the least common multiples of general arithmetic progressions. The bounds of Farhi~\cite{BF1} were then successively improved by Hong and Feng~\cite{8}, Hong and Yang~\cite{HY}, Hong and the second author~\cite{Ko}, and Wu, Tan, and Hong~\cite{Hong}.  Farhi and the first author~\cite{FK} also obtained some related results regarding $\lcm(u_0+1,\ldots, u_0+n)$ that have recently been extended to general arithmetic progressions by Hong and Qian~\cite{HQ}.

In this article, we study finite arithmetic progressions $\{u_k\}_{k=0}^n$, where $u_k:=u_0+kr$ for fixed positive integers $u_0$ and $r$ satisfying $(u_0,r)=1$.  Throughout, we let $n\geq 0$ be a nonnegative integer and define $$L_n:=\lcm(u_0,\ldots, u_n)$$ to be the least common multiple of the sequence $\{u_0,\ldots , u_n\}$.  We are interested in the size of $L_n$ for various choices of the parameters  $u_0$, $r$, and $n$, particularly in the case that $n$ is large relative to $u_0$ and $r$.

The strongest previously known lower bound on $L_n$ is the following result of Wu, Tan, and Hong \cite{Hong}.

\begin{theorem}[{\cite[Thm.~1.3]{Hong}}]\label{oldthm}
Let $a,\ell\geq 2$ be given integers.  Then for any integers $\alpha\geq a$, $r\geq\max(a,\ell-1)$, and $n\geq\ell \alpha r$, we have $L_n \geq u_0 r^{(\ell-1)\alpha +a -\ell}(r+1)^n$.
\end{theorem}

After introducing relevant notation and preliminary results in Section~\ref{sec2},  we prove the following lower bound on $L_n$ in Section~\ref{sec3}.
\begin{theorem}\label{basicbound}
Letting $k$ be an integer with $0\leq k \leq n$, we have
\begin{equation}\label{basiceq}
L_n \geq \frac{u_k  \cdots u_n}{(n-k)!}\prod_{{p\mid r}\atop {p\leq n-k}} \left( \frac{p^{(n-k)/(p-1)}}{n-k+1}\right),
\end{equation}
where the product runs over primes $p\leq n-k$ dividing $r$.
\end{theorem}

In Section \ref{sec35}, we derive several consequences of Theorem \ref{basicbound}.  In particular, we show the following result.
\begin{cor}\label{niceboundcor}If $r>1$ and $k$ is an integer with $0\leq k < n$, then we have that
\begin{equation}\label{nicebound}
L_n \geq r^{\frac{(n-k+1)r-1}{r-1}} \binom{\left(\frac{u_{k-1}}{r}\right) + (n-k+1)}{n-k+1}.
\end{equation}
\end{cor}
Here and hereafter, we define binomial coefficients with non-integral arguments by interpolating the defining factorials using the Gamma function.

We determine the value of $k$ which provides the strongest form of~\eqref{nicebound} in the case that $r$ is prime, and show that in that case Corollary \ref{niceboundcor} improves upon Theorem~\ref{oldthm} whenever $u_0 \gg_{n,r} 1$ or $n \gg r^2$.
Then, in Section~\ref{sec4}, we show that the bound in Corollary \ref{niceboundcor} is sharp up to a factor of $n+1$ for~$u_0$ properly chosen and~$r$ prime.  We study asymptotics for large~$n$ in Section~\ref{sec5}, showing that when~$r$ is prime, \eqref{nicebound} is nearly sharp as~$n\to\infty$ (with $u_0$ and $r$ held fixed).  We conclude  in Section~\ref{sec6}.

As we discuss in Section~\ref{sec6}, our approach extends the methods of Hong and Feng~\cite{8} and the other recent work (\cite{HY}, \cite{Ko}, and \cite{Hong}), pushing these methods nearly to their limits.  The asymptotic estimates we obtain in Section~\ref{rem*} suggest that still better bounds may be possible, but these bounds will likely require new techniques.

\section{Preliminaries}\label{sec2}

Following Hong and Feng~\cite{8} and the subsequent work, we denote, for each integer $0\leq k\leq n$, $$C_{n,k}:=\frac{u_k\cdots u_n}{(n-k)!},\quad L_{n,k}:=\lcm(u_k,\ldots,u_n).$$  From the latter definition, we have that $L_n=L_{n,0}$.

We now note two preliminary lemmata which we use in the sequel.  First, we state the following lemma which first appeared in~\cite{BF1} and has been reproven in several sources.
\begin{lemma}[{\cite[Thm.~2.4]{BF1}, \cite[Thm.~3]{BF2}, \cite[Lem.~2.1]{8}}]\label{L1}
For any integer $n\geq 1$, $L_n=\ell\cdot C_{n,0}$ for some integer~$\ell$.
\end{lemma}
Applying Lemma \ref{L1} to the arithmetic progression $u_k,u_{k+1},\ldots,u_n$, we see that for all $k$ with $0\leq k\leq n$,
 $$L_{n,k}=A_{n,k}\frac{u_k\cdots u_n}{(n-k)!}=A_{n,k}\cdot C_{n,k}$$ for an integer $A_{n,k}\geq 1$.

Now, we introduce a well known lemma regarding the highest power of a prime dividing a factorial.

\begin{lemma}\label{factorialdivide}
If $p$ is a prime and $m\geq 0$ is an integer, then the largest integer, $s$,  so that $p^s\mid m!$ satisfies $$\frac{m}{p-1} > s \geq \frac{m}{p-1} - \log_p(m+1).$$
\end{lemma}
This result is well known; however, we include its proof in Appendix~\ref{app1} for completeness.

\section{Proof of Theorem \ref{basicbound}}\label{sec3}

We begin by noting that
$$
L_n = \lcm(u_0,\ldots,u_n) \geq \lcm(u_k,\ldots,u_n) = L_{n,k}.
$$
We recall that $L_{n,k} = A_{n,k}\cdot C_{n,k}$, where
$$C_{n,k}:=\frac{u_k\cdots u_n}{(n-k)!}$$
and $A_{n,k}$ is an integer.  We notice that any prime $p$ dividing $r$ does not divide $u_k\cdots u_n$.  Therefore, since $L_{n,k}$ is an integer, any power of $p$ dividing $(n-k)!$ must also divide $A_{n,k}$.  By Lemma~\ref{factorialdivide}, we know that $(n-k)!$ is divisible by $p^{a_p}$, with
$$
a_p \geq \frac{n-k}{p-1} - \log_p (n-k+1).
$$
Hence, as $p\mid (n-k)!$ implies that $p\leq n-k$, we have
$$
A_{n,k} \geq \prod_{{p\mid r}\atop {p\leq n-k}} p^{a_p} \geq \prod_{{p\mid r}\atop {p\leq n-k}} \left( \frac{p^{(n-k)/(p-1)}}{n-k+1}\right).
$$
It then follows that \begin{align*}L_n\geq L_{n,k} = C_{n,k}A_{n,k}\geq
\frac{u_k  \cdots u_n}{(n-k)!}\prod_{{p\mid r}\atop {p\leq n-k}} \left( \frac{p^{(n-k)/(p-1)}}{n-k+1}\right),
\end{align*}as in~\eqref{basiceq}.

\section{Consequences of Theorem \ref{basicbound}}\label{sec35}

We begin with the following observation.
\begin{obs}\label{obs1}
The quantity $\frac{x^{(n-k)/(x-1)}}{n-k+1}$ is decreasing in $x$ for $x\geq 2$, and is equal to $1$ when $x=n-k+1$.
\end{obs}
\begin{proof}
The value at $x=n-k+1$ is easily verified. To show that the quantity in question is decreasing for $x\geq 2$, it suffices to show that $x^{1/(x-1)}$ is decreasing for $x\geq 2$.  After taking a logarithm, we see that this is equivalent to showing that $\frac{\log(x)}{x-1}$ is decreasing for $x\geq 2$.

Now, the derivative of $\frac{\log(x)}{x-1}$ is
$$
-\frac{\log(x)}{(x-1)^2} + \frac{1}{x(x-1)} = \frac{x-1-x\log(x)}{x(x-1)^2};
$$
hence, the claim reduces to showing that
\begin{equation}
1+x(\log(x)-1)>0 \quad \text{for all $x\geq 2$.}
\label{deriveq}
\end{equation}
 But \eqref{deriveq} is immediate because $1+x(\log(x)-1)$ is increasing in $x$, and is bigger than $1+2(\frac{1}{2}-1)=0$ for $x=2$.
\end{proof}

We now derive two implications of Theorem~\ref{basicbound}.

\begin{cor}\label{basicbound''}
Letting $k$ be an integer with $0\leq k < n$, we have that
\begin{equation*}
L_n \geq \frac{u_k  \cdots u_n}{(n-k)!}\left( \frac{q^{(n-k)/(q-1)}}{n-k+1}\right),
\end{equation*}
for any prime $q$ dividing $r$.
\end{cor}
\begin{proof}
We see by Observation \ref{obs1} that for primes not equal to $p$, the terms of the product in~\eqref{basiceq} are bigger than $1$.  Thus, we have
\begin{equation}\label{sillyequation}
L_n \geq \frac{u_k  \cdots u_n}{(n-k)!}\prod_{{p\mid r}\atop {p\leq n-k}} \left( \frac{p^{(n-k)/(p-1)}}{n-k+1}\right) \geq \frac{u_k  \cdots u_n}{(n-k)!}\cdot \eta,
\end{equation}
where
$$
\eta=\begin{cases} \frac{q^{(n-k)/(q-1)}}{n-k+1} & q\leq n-k,\\ 1 & \text{ otherwise}.  \end{cases}
$$
As $\eta\geq \frac{q^{(n-k)/(q-1)}}{n-k+1}$ (by Observation \ref{obs1}), \eqref{sillyequation} shows the result.
\end{proof}

\begin{cor}\label{basicbound'}
If $r>1$ and $k$ is an integer with $0\leq k < n$, then we have that
\begin{equation}\label{basiceq'}
L_n \geq \frac{u_k  \cdots u_n}{(n-k)!}\left( \frac{r^{(n-k)/(r-1)}}{n-k+1}\right).
\end{equation}
\end{cor}
\begin{proof}
Letting $q$ be any prime factor of $r$, we have by Corollary \ref{basicbound''} and Observation \ref{obs1} that
\begin{align*}
L_n \geq\frac{u_k  \cdots u_n}{(n-k)!}\left( \frac{q^{(n-k)/(q-1)}}{n-k+1}\right) \geq\frac{u_k  \cdots u_n}{(n-k)!}\left( \frac{r^{(n-k)/(r-1)}}{n-k+1}\right).&\qedhere
\end{align*}
\end{proof}

The bounds of Corollaries~\ref{basicbound''} and~\ref{basicbound'} agree with that of Theorem~\ref{basicbound} when $r$ is prime and at most $n-k$.  Also, rearranging the terms on the right-hand side of~\eqref{basiceq'} yields Corollary~\ref{niceboundcor}.

\begin{proof}[Proof of Corollary \ref{niceboundcor}]
We note that
\begin{align*}
u_k \cdots u_n & = (u_{k-1} + r) \cdots (u_{k-1} + r(n-k+1))\\
& = r^{n-k+1} \left(\frac{u_{k-1}}{r}+1 \right)\cdots\left(\frac{u_{k-1}}{r}+(n-k+1) \right)\\
& = r^{n-k+1} (n-k+1)!  \binom{\left(\frac{u_{k-1}}{r}\right) + (n-k+1)}{n-k+1};
\end{align*}
the result then follows from Corollary \ref{basicbound'}.
\end{proof}
We now determine the value of~$k$ which yields the best bound in Corollary~\ref{niceboundcor}.  It is clear that increasing~$k$ in~\eqref{nicebound} increases the right-hand term of~\eqref{nicebound} by a factor of $$r^{-\frac{r}{r-1}}\left( \frac{n-k+1}{u_kr^{-1}}\right)=\left(\frac{1}{r \cdot r^{1/(r-1)}} \right)\left( \frac{n-k+1}{u_kr^{-1}}\right)=\frac{n-k+1}{u_kr^{1/(r-1)}}.$$  Since this factor is decreasing in $k$, the optimal bound~\eqref{nicebound} is achieved when $$k=k^*:=\max\left\{0,\floor{\frac{n+1-u_0r^{1/(r-1)}}{r^{r/(r-1)}+1}}\right\}.$$

\subsection*{Remarks}

The Wu, Tan, and Hong \cite{Hong} proof of Theorem \ref{oldthm} follows from establishing the inequality
\begin{align}\label{basicHongBound1}
L_n &\geq \frac{u_k\cdots u_n}{(n-k)!} \cdot r^{\floor{(n-k)/r}} \\&=C_{n,k}\cdot r^{\floor{(n-k)/r}} \\
&\geq \left(u_0(r+1)^{n}\right)r^{\floor{(n-k)/r}}\label{basicHongBound2}
\end{align}
and then taking
\begin{equation}\label{LKL}k=\max\left\{0,\left\lfloor \frac{n-u_0}{r+1}\right\rfloor+1\right\}\approx \frac{n}{r+1}.\end{equation} The exact bound in Theorem~\ref{oldthm} follows from \eqref{basicHongBound1}--\eqref{basicHongBound2} because, as Wu, Tan, and Hong \cite{Hong} show,
$$
\left(u_0(r+1)^{n}\right)r^{\floor{(n-k)/r}} \geq u_0 r^{(\ell-1)\alpha +a -\ell}(r+1)^n
$$ for $a$, $\ell$, and $\alpha$ satisfying the hypotheses of Theorem \ref{oldthm}.

We improve upon Theorem~\ref{oldthm} in several ways.  First, our bound in Corollary \ref{niceboundcor} is sharper than the inequality in~\eqref{basicHongBound1} for $n\gg r^2$. Indeed, the right-hand side of \eqref{nicebound} is equal to $\frac{u_k\cdots u_n}{(n-k)!} \cdot r^{\floor{(n-k)/r}}$ up to a power of $r$. But the power appearing in  \eqref{nicebound} is proportional to $\frac{n}{r-1}$, rather than $\frac{n}{r}$. Second, we leave our bound in its native form, rather than weakening it by replacing $C_{n,k}$ by $u_0(r+1)^n$ as in \eqref{basicHongBound2}.  This latter improvement is particularly significant for $u_0$ large. In particular, for fixed $n$, and $r$, we have $C_{n,k}$ proportional to $u_0^{n-k}$, which is much greater than $u_0(r+1)^n$ when $u_0$ is large.
Finally, we use $k^*$, which optimizes our bound, instead of using the value of $k$ employed by Wu, Tan, and Hong~\cite{Hong}. With $k$ as in~\eqref{LKL}, if $n\gg r^2$ or $u_0 \gg_{n,r} 1$, we have
\begin{align}\label{comparisonEquation}
r^{\frac{(n-k^*+1)r-1}{r-1}} \binom{\left(\frac{u_{k^*-1}}{r}\right) + (n-k^*+1)}{n-k^*+1} & \geq r^{\frac{(n-k+1)r-1}{r-1}} \binom{\left(\frac{u_{k-1}}{r}\right) + (n-k+1)}{n-k+1}\notag \\ & \gg \left(u_0(r+1)^{n}\right)r^{\floor{(n-k)/r}}\notag\\
& \geq u_0 r^{(\ell-1)\alpha +a -\ell}(r+1)^n.
\end{align}
We see the bound obtained in Corollary \ref{niceboundcor} (which is given by the left-hand side of \eqref{comparisonEquation}) is larger than the bound of Theorem \ref{oldthm} (which is given by the right-hand side of \eqref{comparisonEquation}). Furthermore, this difference is significant when $n\gg r^2$ or $u_0 \gg_{n,r} 1$.

\section{Bounds for Large $u_0$}\label{sec4}

When $u_0>n$, we have $k^*=0$ and therefore get the best bound from Corollary~\ref{niceboundcor} by setting~$k=0$ in~\eqref{nicebound}.  This indicates that the following consequence of Corollary \ref{basicbound'} is sharpest for large $u_0$.
\begin{cor}\label{corubig}
If $r>1$, then we have that
\begin{equation}\label{niceboundubig}
L_n \geq r^{\frac{(n+1)r-1}{r-1}} \binom{\left(\frac{u_{-1}}{r}\right) + n+1}{n+1} = \frac{u_0\cdots u_n}{n!} \left( \frac{r^{\frac{n}{r-1}}}{n+1}\right).
\end{equation}
\end{cor}

For appropriately chosen $u_0$, and $r$ prime, the bound~\eqref{niceboundubig} of Corollary~\ref{corubig} is sharp to within a factor of $n+1$.

\begin{obs}\label{obs2}
If $r$ is prime and $u_0$ is divisible by the prime-to-$r$ part of $n!$, then~\eqref{niceboundubig} is tight up to a factor of $n+1$.
\end{obs}
\begin{proof}
Let $N$ be the prime-to-$r$ part of $n!$ and observe that by Lemma \ref{factorialdivide}, $N > n! r^{-\frac{n}{r-1}}$.  Hence it suffices to show that
$$
\tilde L :=\frac{u_0\cdots u_n}{N} \geq L_n.
$$

We claim that $\tilde L$
is a common multiple of $\{u_0,\ldots,u_n\}$.  To see this, we note that since $N \mid u_0$, we have that $\tilde L$ is a multiple of $u_i$ for $1\leq i\leq n$.  Furthermore,
$$
u_1\cdot u_2 \cdots u_n \equiv (r)(2r)\cdots (nr) \equiv n! r^n \equiv 0\bmod{N}.
$$
Thus $\frac{u_1\cdots u_n}{N}$ is an integer, and hence $u_0 \mid \tilde L$.  Thus $\tilde L$ is a common multiple of $\{u_0,\ldots,u_n\}$ and is therefore larger than $L_n=\lcm(u_0,\ldots,u_n).$
\end{proof}

\section{Asymptotics for Large $n$}\label{sec5}

We  now determine the asymptotics of the lower bound~\eqref{nicebound} of Corollary~\ref{niceboundcor} when~$n$ is large relative to~$u_0$ and~$r>1$.  We notice that for~$n$ large and~$k$ within some (additive) constant $\kappa$ of its optimal value,~$k^*$, the multiplicative change in~\eqref{nicebound} is $(1+o_{u_0,r,\kappa}(1))$, where $o_{u_0,r,\kappa}(1)$ denotes some function of $n$, $u_0$, $\kappa$, and $r$ that has limit $0$ whenever $u_0$, $r$, and $\kappa$ are held constant and $n\to\infty$.  Furthermore, as the binomial coefficient in~\eqref{nicebound} is interpolated using the Gamma function, this  observation holds even for fractional values of $k$.
\begin{obs}\label{logsecderLem}
Let
$$
f(n,k) = f_{u_0,r}(n,k) := r^{\frac{(n-k+1)r-1}{r-1}} \binom{\left(\frac{u_{k-1}}{r}\right) + (n-k+1)}{n-k+1}.
$$
Then, for $\vert k-k^*\vert <\kappa$, we have that
$$
\frac{f(n,k)}{f(n,k^*)} = 1 +o_{u_0,r,\kappa}(1).
$$
\end{obs}
\begin{proof}
First, we note that $\log(f(n,k))$ is a smooth function in $k$. As $\log(f(n,k^*))>\log(f(n,k^*\pm 1))$, we see that $\log(f(n,k))$ must have derivative $0$ at some $k=\tilde k$ with $\vert k^*-\tilde k\vert \leq 1.$ We show that for all $\vert k-\tilde k\vert <\kappa+1$,
$$
\frac{f(n,k)}{f(n,\tilde k)} = 1 +o_{u_0,r,\kappa}(1).
$$
To show this, it is sufficient to show that the second derivative of $\log(f(n,k))$ is $o_{u_0,r,\kappa}(1)$ for all $k$ with $\vert k-\tilde k\vert < \kappa+1$. To see this, we observe that the logarithmic second derivative of $r^{\frac{(n-k+1)r-1}{r-1}}$ is trivial, while the logarithmic second derivative of $ \binom{\left(\frac{u_{k-1}}{r}\right) + (n-k+1)}{n-k+1}$ is the negative of the sum of the logarithmic second derivatives of $\Gamma$ at $n-k+2$ and $\frac{u_{k-1}}{r}+1$. Thus, the result follows from the fact that $\frac{\partial^2}{\partial x^2}\log(\Gamma(x)) \rightarrow 0$ as $x\rightarrow \infty.$
\end{proof}

By Observation \ref{logsecderLem}, we get asymptotically equivalent bounds (for fixed $u_0$ and $r$, as $n\to\infty$) if we consider~\eqref{nicebound} with any $k$ within $O_{u_0,r}(1)$ of $k^*$.

Now, we set $$\newk:=1+\frac{n}{r^{r/(r-1)}+1}-\frac{u_0}{r(r^{-r/(r-1)}+1)},$$ noting that $\newk$ is within $O_{u_0,r}(1)$ of~$k^*$ for all~$n$.  We set $$\ratio:= r^{-r/(r-1)}=\frac{\left(\frac{u_{\newk-1}}{r}\right) + (n-\newk+1)}{n-\newk+1}-1,$$ so that if we take $k=\newk$ in~\eqref{nicebound}, the ratio of the terms in the binomial coefficient  equals $\ratio+1$.  For ease of notation, we also denote $$\numerator:=\left(\frac{u_{\newk-1}}{r}\right) + (n-\newk+1)=\frac{u_n}{r},$$ so that the binomial coefficient in~\eqref{nicebound} with~$k=\newk$ is \begin{equation}\label{stirmeplease}\binom{\mu}{\mu/(\ratio+1)}.\end{equation}  By Stirling's formula,~\eqref{stirmeplease} is $$\frac{1+\ratio}{\sqrt{2\pi\mu\ratio}}\left((1+\ratio)^{\frac{1}{1+\ratio}} \left(\frac{1+\ratio}{\ratio}\right)^{\frac{\ratio}{1+\ratio}}\right)^{\mu}(1+o_{u_0,r}(1)).$$  It follows that our lower bound is asymptotic to \begin{equation}
	 r^{\frac{(n-\newk+1)r-1}{r-1}}\left(\frac{1+\ratio}{\sqrt{2\pi\mu\ratio}}\right) \left((1+\ratio)^{\frac{1}{1+\ratio}}\left(\frac{1+\ratio}{\ratio}\right)^{\frac{\ratio}{1+\ratio}}\right)^{\mu}(1+o_{u_0,r}(1)).
	\label{asbound}
\end{equation}
The exponential part of~\eqref{asbound} is \begin{equation}\label{exppart}\left(r^{\frac{r}{(1+\ratio)(r-1)}}(1+\ratio)^{\frac{1}{1+\ratio}}\left(\frac{1+\ratio}{\ratio}\right)^{\frac{\ratio}{1+\ratio}}\right)^{n}.\end{equation}

Bateman, Kalb, and Stenger \cite{BKS} computed the asymptotics of the least common multiple of a long sequence of consecutive integers, deriving an asymptotic formula for $\log(L_n)$ for fixed $u_0$ and~$r$. Now, for completeness, we reproduce the \cite{BKS} asymptotic before comparing it with our bound \eqref{asbound}.

We note that
$$
\log(L_n) = \sum_{d\mid L_n} \Lambda(d),
$$
where $\Lambda(d)$ is the Von Mangoldt function.  By definition, $\Lambda(d)$ is $0$ unless~$d$ is a power of a prime.  Furthermore, for $d$ a power of a prime, $d\mid L_n$ if and only if~$d\mid u_k$ for some $k$ ($0\leq k\leq n$).  Therefore we have that
\begin{equation}\label{step1}
\log(L_n) = \sum_{{d\mid u_k}\atop{\text{for some $0\leq k\leq n$}}} \Lambda(d).
\end{equation}

We claim that if~$n$ is sufficiently large, $L_n$ is divisible by all of the finitely many positive integers less than $u_0$ and congruent to $u_0$ modulo~$r$.  In particular, if $n>ru_0^2$ and $u_0>u>0$ with $u \equiv u_0 \bmod r$, then $u(ru_0+1)$ divides $L_n$, and thus so does $u$.  For such $n$, the~$d$ in~\eqref{step1} are exactly the~$d$ dividing some positive integer $u\leq u_n$ with $u\equiv u_0\bmod r$.  Clearly the smallest positive integer congruent to $u_0$ modulo~$r$ and divisible by $d$ is $d\cdot \ell_d$, where $\ell_d$ is the smallest positive representative of the conjugacy class of $\frac{u_0}{d}$ modulo~$r$.  Hence, we may break up the sum in~\eqref{step1} to obtain
\begin{equation}\label{step2}
\log(L_n) = \sum_{{(\ell,r)=1}\atop{0<\ell\leq r}} \sum_{{d<\frac{u_n}{\ell}}\atop{ d\equiv \frac{u_0}{\ell}\bmod{r}}} \Lambda(d).
\end{equation}
We recall that the inner sum in~\eqref{step2} is $\left(\frac{1}{\varphi(r)}\right)\left(\frac{u_n}{\ell}\right)(1+o_{u_0,r}(1))$, where $\varphi$ is the Euler totient function (see~\cite[p.~122, eq.~(5.71)]{IK}). Therefore, we have that
\begin{equation}\label{step3}
\log(L_n) = \frac{u_n}{\phi(r)}\left(\sum_{{(\ell,r)=1}\atop{0<\ell\leq r}}\frac{1}{\ell}\right)(1+o_{u_0,r}(1)).
\end{equation}

If we assume that $r$ is prime, then~\eqref{step3} reduces to \begin{equation}\label{step4} \log(L_n) = \frac{u_n}{r-1}H_{r-1}(1+o_{u_0,r}(1)), \end{equation} where $H_{r-1}$ denotes the $(r-1)$-st harmonic number.

\subsection*{Remarks}\label{rem*}

We note that our proven asymptotic for $\log(L_n)$ has linear term $$n\left(\frac{rH_{r-1}}{r-1} \right)=n\left(\log(r) + \gamma +O\left(\frac{\log(r)}{r} \right) \right),$$  where $\gamma$ is the Euler-Mascheroni constant. The asymptotic lower bound~\eqref{asbound} we prove has exponential term~\eqref{exppart} with logarithm
$$
n\left(\frac{r\log(r) }{(r-1)(\ratio+1)} + \frac{\log(1+\ratio)}{1+\ratio} + \left(\frac{\ratio}{1+\ratio} \right)\log\left( \frac{1+\ratio}{\ratio}\right) \right)=n\left( \log(r) + O\left(\frac{\log(r)}{r}\right) \right),
$$
as we have $\ratio = O\left( \frac{1}{r}\right)$. Thus, we see that our bound~\eqref{nicebound} of Corollary~\ref{niceboundcor} is within a multiplicative factor of $$e^{\gamma n  ( 1 + o_{u_0,r}(1) + O(\log(r)/r))}$$ of being sharp.  In particular, we have for any fixed $u_0$ that
$$
\lim_{\substack{{r\rightarrow\infty}\\ {r \text{ prime}}}} \lim_{\substack{n\rightarrow \infty\\ {}}} \left(\frac{r^{\frac{(n-k^*+1)r-1}{r-1}} \left({\left(\frac{u_{k^*-1}}{r}\right) + (n-k^*+1)}\atop{n-k^*+1}\right)}{L_n} \right)^{1/n} = e^{-\gamma}.
$$

\section{Conclusion}\label{sec6}

Determining lower bounds on $L_n$ is clearly equivalent to the problem of finding lower bounds for $A_{n,k}$.  We have so far obtained these bounds by noting that, although $L_{n,k}$ is always an integer, $C_{n,k}$ need not be integral.  In essence, this is the same strategy which has been applied in the work of Hong and Feng~\cite{8}, Hong and Yang~\cite{HY}, Hong and the second author~\cite{Ko}, and Wu, Tan, and Hong~\cite{Hong}. In this article, we have pushed these techniques nearly to their limits.  It is relatively easy to show that $C_{n,k}$ does not have in its denominator any prime factors which do not also divide $r$.  Furthermore, we have accounted almost exactly for the contributions of these primes to the denominator of~$C_{n,k}$.  Hence, further progress towards bounding $L_n$ should come from new techniques for bounding $A_{n,k}$.

Fortunately, there is hope that better bounds on $A_{n,k}$ can be obtained.  The proof that~$C_{n,k}$ divides $L_{n,k}$ considers the potential common divisors of the elements $\{u_k,\ldots,u_n\}$.  On the other hand, unless $u_k$ is chosen very carefully, not all of these common divisors actually appear.  In particular, for $A_{n,k}$ to have no factors prime to $r$, it  needs to be the case that the prime-to-$r$ part of $n-k-m$ divides $u_k\cdots u_{k+m}$ for each $m$.  For each such divisibility condition that fails, we gain extra factors for $A_{n,k}$.  Furthermore, we know that such factors must exist since (as was shown in Section~\ref{rem*}), for large $n$ and prime $r$, our bound fails by a factor of roughly $e^{\gamma n}$.

\section*{Acknowledgements}

While working on this paper, the first author was partly supported by NDSEG and NSF Graduate Research Fellowships, and the second author was partly supported by an NSF Graduate Research Fellowship, NSF grant CCF-1216095, and an AMS-Simons Travel Grant.  Both authors thank Noam~D.\ Elkies, Benedict~H.\ Gross, and an anonymous referee for helpful comments and suggestions.

\appendix
\section{Proof of Lemma~\ref{factorialdivide}}\label{app1}

For each $k>1$ there are $\floor{\frac{m}{p^k}}$ integers in $1,2,\ldots,m$ divisible by $p^k$.  Together these produce all the factors of $p$ dividing $m!$.  Thus
$$
s = \sum_{k=1}^{\infty} \floor{\frac{m}{p^k}} < \sum_{k=1}^\infty \frac{m}{p^k} = \frac{m}{p-1}.
$$
It follows easily by induction upon $m$ that $\sum_{k=1}^\infty \floor{\frac{m}{p^k}} = \frac{m-d}{p-1}$, where $d$ is the sum of the digits in the base-$p$ representation of $m$.
Thus, we need only show that \begin{equation}
	\log_p(m+1) \geq \frac{d}{p-1}.
	\label{foo!}
\end{equation}

To prove~\eqref{foo!}, we first fix the value of $d$.  We note that the smallest value of $m$ that attains this value of $d$ occurs when all of the base-$p$ digits of $m$ are $p-1$, except for the leading digit, which is, say, $\ell$ ($1\leq \ell \leq p-1$).  We  then have that $m+1 = p^w(\ell+1)$ and $d=w(p-1) + \ell$ for some $w$ and $\ell$ such that $1\leq \ell \leq p-1$.  We need to show that
$$
w+\log_p(\ell+1) = \log_p(p^w(\ell+1)) \geq \frac{w(p-1)+\ell}{p-1} = w + \frac{\ell}{p-1}.
$$
Canceling the additive terms of $w$ on each side, all that is left to prove is that \begin{equation}
	\log_p(\ell+1) \geq \frac{\ell}{p-1}.
	\label{laststepinlem}
\end{equation}  But~\eqref{laststepinlem} follows from the concavity of the logarithm function, since equality holds in~\eqref{laststepinlem} for $\ell=0$ and for $\ell=p-1$.
\bibliographystyle{amsalpha}
\bibliography{lcm_bound_bib}
\end{document}